\documentclass[11pt]{amsart}
\usepackage{amsfonts,amssymb,amsthm}
\usepackage{amsmath,amscd}
\usepackage{pstricks}
\usepackage{pstricks,pst-node}
\usepackage{mathrsfs}
\usepackage[all]{xy}

\DeclareMathAlphabet{\mathpzc}{OT1}{pzc}{m}{it}

\theoremstyle{plain}
\newtheorem{Thm}{Theorem}[section]
\newtheorem{Prop}[Thm]{Proposition}
\newtheorem{Lem}[Thm]{Lemma}
\newtheorem{Coro}[Thm]{Corollary}
\theoremstyle{definition}

\numberwithin{equation}{section}

\newcommand{\ddbfHa}{\dot{\boldsymbol{\mathfrak D}_\vtg}(n)}
\newcommand{\ddbfHasl}{\dot{\boldsymbol{\mathfrak D}'_\vtg}(n)}

\newcommand{\etam}{\eta_m}
\newcommand{\etak}{\eta_k}

\newcommand{\etakt}{\eta_{2k}}

\def\su#1{^{#1}}
\newcommand{\Kbfj}{K^{\bfj}}

\newcommand{\tA}{{}^t\!A}
\newcommand{\tAm}{{}^t\!(A^-)}

\newcommand{\afUslp}{U(\widehat{\frak{sl}}_n)^+_\sZ}
\newcommand{\afUslNp}{U(\widehat{\frak{sl}}_N)^+_\sZ}
\newcommand{\afbfUslNp}{{\bf U}(\widehat{\frak{sl}}_N)^+}

\newcommand{\bin}{\bigcup}

\newcommand{\han}{\subseteq}

\newcommand{\bsS}{{\boldsymbol{\sS}}}

\newcommand{\lan}{\langle}
\newcommand{\ran}{\rangle}

\newcommand{\leb}{\left[}

\newcommand{\rib}{\right]}

\def\lr#1{\langle #1\rangle}

\def\ggp#1#2{\left[\kern-3.2pt\left[{#1\atop #2}\right]\kern-3.2pt\right]}

\newcommand{\Boleq}{\leq^{Bo}}
\newcommand{\Bol}{<^{Bo}}

\def\fS{{\frak S}}

\def\fb{{\frak b}}
\def\fc{{\frak c}}

\newcommand{\g}{{\mathsf g}}
\newcommand{\sff}{{\mathsf f}}
\newcommand{\sfh}{{\mathsf h}}

\newcommand{\msL}{\mathscr L}
\newcommand{\msR}{\mathscr R}

\newcommand{\msD}{\mathscr D}\newcommand{\afmsD}{{\mathscr D}^\vtg}

\newcommand{\affSr}{{\fS_{\vtg,r}}}
\newcommand{\Wr}{W_r}

\newcommand{\afHr}{{\sH_\vtg(r)_\sZ}}

\def\sH{{\mathcal H}}

\def\sO{{\mathcal O}}

\def\sS{{\mathcal S}}
\def\sT{{\mathcal T}}

\def\sX{{\mathcal X}}

\def\sZ{{\mathcal Z}}

\newcommand{\vtg}{{\!\vartriangle\!}}

\newcommand{\dbfHa}{{\boldsymbol{\mathfrak D}_\vtg}(n)}
\newcommand{\dbfHasl}{{\boldsymbol{\mathfrak D}'_\vtg}(n)}

\newcommand{\dHa}{{{\mathfrak D}_\vtg}(n)_\sZ}
\newcommand{\dHapz}{{{\mathfrak D}_\vtg}(n)^{\geq 0}_\sZ}

\newcommand{\dHap}{{{\mathfrak D}_\vtg}(n)^+_\sZ}

\newcommand{\dHaz}{{{\mathfrak D}_\vtg}(n)^0_\sZ}

\def\field{{\mathbb F}}

\def\deg{{\rm deg}}

\newcommand{\mnmod}{\!\!\!\mod\!}

\newcommand{\tri}{\triangle(n)}

\newcommand{\afsl}{\widehat{\frak{sl}}_n}
\newcommand{\afgl}{\widehat{\frak{gl}}_n}

\newcommand{\afFn}{{\mathscr F_\vtg}}

\newcommand{\afE}{E^\vartriangle}

\newcommand{\afSr}{{\mathcal S}_{\vtg}(n,r)_\sZ}

\newcommand{\afSrpz}{{\mathcal S}_{\vtg}(n,r)^{\geq 0}_\sZ}

\newcommand{\afbfSr}{{\boldsymbol{\mathcal S}}_\vtg(n,r)}
\newcommand{\afbfSNr}{{\boldsymbol{\mathcal S}}_\vtg(N,r)}
\newcommand{\afbfSrn}{{\boldsymbol{\mathcal S}}_\vtg(n,r+n)}

\newcommand{\afal}{{\boldsymbol\alpha}^\vartriangle}

\newcommand{\afbse}{\boldsymbol e^\vartriangle}
\newcommand{\afmbnn}{\mathbb N_\vtg^{n}}
\newcommand{\afmbzn}{\mathbb Z_\vtg^{n}}

\newcommand{\afLa}{\Lambda_\vtg}
\newcommand{\afLanr}{\Lambda_\vtg(n,r)}

\newcommand{\afThn}{\Theta_\vtg(n)}
\newcommand{\aftiThn}{\ti\Theta_\vtg(n)}

\newcommand{\afThnpm}{\Theta_\vtg^\pm(n)}
\newcommand{\afThnp}{\Theta_\vtg^+(n)}
\newcommand{\afThNp}{\Theta_\vtg^+(N)}

\newcommand{\afThnr}{\Theta_\vtg(n,r)}
\newcommand{\afThNr}{\Theta_\vtg(N,r)}
\newcommand{\afTh}{\Theta_\vtg}

\newcommand{\dzr}{\dot{\zeta}_r}
\newcommand{\dz}{\dot{\zeta}}

\newcommand{\dxr}{\dot{\xi}_r}
\newcommand{\dx}{\dot{\xi}}

\newcommand{\tu}{\widetilde u}

\def\leq{\leqslant}\def\geq{\geqslant}
\def\ge{\geqslant}

\newcommand{\iy}{\infty}

\newcommand{\Th}{\Theta}
\newcommand{\dt}{\delta}

\newcommand{\lm}{\longmapsto}

\newcommand{\vi}{\varphi}

\newcommand{\up}{v}
 
 \newcommand{\ep}{\varepsilon}
  
 \newcommand{\al}{\alpha}
 \newcommand{\bt}{\beta}
 \newcommand{\h}{\widehat}
 \newcommand{\ti}{\widetilde}

\newcommand{\zr}{\zeta_r}

\newcommand{\sg}{\sigma}

\def\th{\theta}

\newcommand{\p}{\prec}
\newcommand{\pr}{\preccurlyeq}
\newcommand{\bop}{\bigoplus}

\newcommand{\ot}{\otimes}

\newcommand{\bfl}{\mathbf{0}}

\newcommand{\Ar}{{A,r}}

\newcommand{\ol}{\overline}
\newcommand{\lra}{\longrightarrow}
\newcommand{\ra}{\rightarrow}
 \newcommand{\la}{{\lambda}}

 \newcommand{\La}{\Lambda}

 \newcommand{\mbn}{\mathbb N}
 \newcommand{\mbq}{\mathbb Q}
 
 \newcommand{\mbz}{\mathbb Z}

  \newcommand{\bfd}{{\mathbf{d}}}
 
 \newcommand{\bfj}{{\mathbf{j}}}

\newcommand{\bfa}{{\boldsymbol{a}}}
\newcommand{\bfb}{{\boldsymbol{b}}}

\newcommand{\bfU}{{\mathbf{U}}}
\newcommand{\bfL}{{\mathbf{L}}}

\newcommand{\ga}{{\gamma}}

\newcommand{\bfB}{\mathbf{B}}
\newcommand{\bfBn}{\mathbf{B}(n)}
\newcommand{\bfBN}{\mathbf{B}(N)}
\newcommand{\bfBNap}{\mathbf{B}(N)^{\mathrm{ap}}}
\newcommand{\dbfBn}{\dot{\mathbf{B}}(n)}
\newcommand{\bfBr}{\mathbf{B}(n,r)}

\newcommand{\End}{\operatorname{End}}
\newcommand{\Hom}{\operatorname{Hom}}

\newcommand{\spann}{\operatorname{span}}
\newcommand{\diag}{\operatorname{diag}}

\def\ro{\text{\rm ro}}
\def\co{\text{\rm co}}

\def\afsygr{{\fS_{\vtg,r}}}

\newcommand{\dbfU}{\dot{{\bfU}}(\h{\frak{sl}}_n)}

\newcommand{\afThnap}{\Theta_\vtg(n)^{\rm ap}}
\newcommand{\afThnpap}{\Theta_\vtg^+(n)^{\mathrm{ap}}}
\newcommand{\afThNpap}{\Theta_\vtg^+(N)^{\mathrm{ap}}}
\newcommand{\afThnrap}{\Theta_\vtg(n,r)^{\rm ap}}
\newcommand{\afThnrzap}{\Theta_\vtg(n,r_0)^{\rm ap}}
\newcommand{\afThnrnap}{\Theta_\vtg(n,r+n)^{\rm ap}}
\newcommand{\afThNrap}{\Theta_\vtg(N,r)^{\rm ap}}

\begin{document}
\title{Positivity properties for canonical bases of modified quantum affine $\frak{sl}_n$}

\author{Qiang Fu$^\dagger$}
\address{Department of Mathematics, Tongji University, Shanghai, 200092, China.}
\email{q.fu@hotmail.com, q.fu@tongji.edu.cn}
\author{Toshiaki Shoji}
\address{Department of Mathematics, Tongji University, Shanghai, 200092, China.}
\email{shoji@tongji.edu.cn, shoji@math.nagoya-u.ac.jp}

\thanks{$^\dagger$ Supported by the National Natural Science Foundation
of China, Fok Ying Tung Education Foundation.}

\begin{abstract}
The positivity property for canonical bases asserts that the structure constants of
the multiplication for the canonical basis are in $\mbn[\up,\up^{-1}]$.
Let $\bfU$ be the quantum group over $\mbq(\up)$ associated with a symmetric Cartan datum.
The positivity property for the positive part $\bfU^+$ of $\bfU$ was proved by Lusztig.
He conjectured that the positivity property holds
for the modified form $\dot\bfU$ of $\bfU$.
In this paper, we prove that the structure constants for the canonical basis of  $\dbfU$
coincide with certain structure constants for the canonical basis  of $\afbfUslNp$ for $n<N$.
In particular, the positivity property for $\dbfU$ follows from the positivity property for $\afbfUslNp$.
\end{abstract}
 \sloppy \maketitle
\section{Introduction}

Let $\bfU$ be the quantum group over $\mbq(\up)$ associated with a Cartan datum $(I,\cdot)$, where $\up$ is an indeterminate. It is known by Lusztig and Kashiwara that the positive part $\bfU^+$ of a quantum enveloping algebra $\bfU$ has a canonical basis with remarkable properties (see Kashiwara \cite{Kas1}, Lusztig \cite{Lu90b,Lu91,Lubk}). Among them, the deepest one should be
the positivity property for the canonical basis of $\bfU^+$ proved by Lusztig \cite{Lu90b,Lu91}, \cite[14.4.13]{Lubk}, which asserts that the structure constants of the multiplication for the canonical basis of $\bfU^+$ are in $\mbn[\up,\up^{-1}]$  in the case where the
Cartan datum $(I,\cdot)$ is symmetric.

Let $\dot\bfU$ be the modified form of $\bfU$. The algebra $\dot\bfU$ is an associative algebra without unity and the category of $\bfU$-modules of type $1$ is equivalent to the category of unital $\dot\bfU$-modules. The canonical basis $\dot\bfB$ of $\dot\bfU$ was constructed by Lusztig \cite{Lu92b,Lubk}. In \cite[Section 11]{Lu92b} and \cite[25.4.2]{Lubk},  he conjectured that the structure constants of the multiplication for $\dot\bfB$ are in $\mbn[\up,\up^{-1}]$, i.e., the positivity property holds for  $\dot\bfU$,
in the case where the Cartan datum $(I,\cdot)$ is symmetric.

Let $\afbfSr$ be the affine quantum Schur algebra over $\mbq(\up)$ (see \cite{GV}, \cite{Gr99} and \cite{Lu99}). An explicit algebra homomorphism $\zeta_r$ from  $\bfU(\afsl)$ to $\afbfSr$ was constructed by Ginzburg--Vasserot \cite{GV}, Lusztig \cite{Lu99}.
According to \cite[8.2]{Lu99} the map $\zeta_r:\bfU(\afsl)\ra\afbfSr$ is not surjective in the case where $n\leq r$. In turn, it is proved by Deng--Du--Fu \cite[3.8.1]{DDF} that the map $\zeta_r$ can be extended to a surjective algebra homomorphism  from $\bfU(\afgl)$ to $\afbfSr$, where $\bfU(\afgl)$ is the quantum loop algebra of $\afgl$.
On the other hand, the quantum Schur algebra
$\bsS(n,r)$ is  known to be a quotient of the quantum  algebra $\bfU(\frak{sl}_n)$. The canonical basis of $\bsS(n,r)$ was defined by Beilinson--Lusztig--MacPherson \cite{BLM} and  the positivity property for the canonical basis of $\bsS(n,r)$ was proved by Green in \cite{Gr97}.
The canonical basis $\bfBr$ of the affine quantum Schur algebra $\afbfSr$ was defined in \cite{Lu99}. Lusztig gave in \cite[4.5]{Lu99} a sketch of the proof of the positivity property for $\bfBr$ based on the property of Kazhdan--Lusztig basis of affine Hecke algebras of type $A$.

In this paper, we show that there exist good relations among canonical bases of the three  algebras
$\dbfU$, $\afbfSr$ and $\afbfUslNp$.
In Theorem \ref{prop of canonical basis for affine q-Schur algebras} we prove that the structure constants for $\bfBr$ are determined by the structure constants for the canonical basis $\bfBNap$ of $\afbfUslNp$ for $n<N$. Then the positivity property for $\bfBr$ follows from the positivity property for $\bfBNap$. This gives an alternate approach for the positivity property of $\bfBr$. Using Theorem \ref{prop of canonical basis for affine q-Schur algebras},
we prove in Theorem \ref{relation dbfBn bfBNap} that the
structure constants for the canonical basis $\dbfBn$ of $\dbfU$ are  determined by the structure constants for the canonical basis $\bfBNap$ of
$\afbfUslNp$ for $n<N$. Thus the positivity property for $\dbfBn$ follows from the positivity property for $\bfBNap$.
We also discuss in Theorem \ref{positive modified affine gln} a certain weak positivity property for $\ddbfHa$, where $\ddbfHa$ is the modified quantum affine $\frak{gl}_n$.

{\bf Notation:} For a positive integer $n$, let
$\afThn$  (resp., $\aftiThn$) be the set of all matrices
$A=(a_{i,j})_{i,j\in\mbz}$ with $a_{i,j}\in\mbn$ (resp. $a_{i,j}\in\mbz$, $a_{i,j}\geq0$ for all $i\neq j$)  such that
\begin{itemize}
\item[(a)]$a_{i,j}=a_{i+n,j+n}$ for $i,j\in\mbz$; \item[(b)] for
every $i\in\mbz$, both sets $\{j\in\mbz\mid a_{i,j}\not=0\}$ and
$\{j\in\mbz\mid a_{j,i}\not=0\}$ are finite.
\end{itemize}
Let $\afThnp=\{A\in\afThn\mid a_{i,j}=0\text{ for }i\geq j\}.$
For $r\geq 0$,
let $\afThnr=\{A\in\afThn\mid\sg(A)=r\},$
where $\sg(A)=\sum_{1\leq i\leq n,\,
j\in\mbz}a_{i,j}.$ For $i,j\in\mbz$ let $\afE_{i,j}\in\afThn$ be the matrix
$(e^{i,j}_{k,l})_{k,l\in\mbz}$ defined by
\begin{equation*}e_{k,l}^{i,j}=
\begin{cases}1&\text{if $k=i+sn,l=j+sn$ for some $s\in\mbz$,}\\
0&\text{otherwise}.\end{cases}
\end{equation*}

Let $\afmbzn=\{(\la_i)_{i\in\mbz}\mid
\la_i\in\mbz,\,\la_i=\la_{i-n}\ \text{for}\ i\in\mbz\}$ and $\afmbnn=\{(\la_i)_{i\in\mbz}\in \afmbzn\mid \la_i\ge0\text{ for  }i\in\mbz\}.$
$\afmbzn$ has a natural structure of abelian group. For $r\geq 0$  let
$\afLanr=\{\la\in\afmbnn\mid\sg(\la)=r\},$
where $\sg(\la)=\sum_{1\leq i\leq n}\la_i$.

Let $\sZ=\mbz[\up,\up^{-1}]$, where $\up$ is an indeterminate.

\section{Preliminaries}
\subsection{}
Let $\tri$ ($n\geq 2$) be
the cyclic quiver
with vertex set $I=\mbz/n\mbz$ and arrow set
$\{i\to i+1\mid i\in I\}$. We identify $I$ with $\{1,2,\cdots,n\}$.
Let $\field$ be a field. For $i\in I$ and $j\in\mbz$ with $i<j$, let $S_i$
denote the one-dimensional representation of $\tri$ with $(S_i)_i=\field$ and $(S_i)_k=0$ for $i\neq k$ and
$M^{i,j}$ the unique indecomposable nilpotent representation of length $j-i$ with top $S_i$.

For $A\in\afThnp$ let $\bfd(A)\in\mbn I$ be the dimension vector of $M(A)$,
where
$$M(A)=M_\field(A)=\bop_{1\leq i\leq n\atop i<j,\,j\in\mbz}a_{i,j}M^{i,j}.$$
We will identify naturally $\mbn I$ with $\afmbnn$.
The Euler form associated with the cyclic quiver $\tri$ is the
bilinear form $\lan-,-\ran$: $\afmbzn\times\afmbzn\ra\mbz$ defined by
$\lan\la,\mu\ran=\sum_{1\leq i\leq n}\la_i\mu_i-\sum_{1\leq i\leq n}\la_i\mu_{i+1}$
for $\la,\mu\in\afmbzn$.

By Ringel \cite{Ri93}, for $A,B,C\in\afThnp$,
there is a polynomial $\vi^{C}_{A,B}\in\mbz[\up^2]$  such
that, for any finite field $\field_q$,
$\vi^{C}_{A,B}|_{\up^2=q}$ is equal to the number of submodules $N$ of
$M_{\field_q}(C)$ satisfying $N\cong M_{\field_q}(B)$ and $M_{\field_q}(C)/N\cong M_{\field_q}(A)$.

Let $\dbfHa$ be the double Ringel--Hall algebra of the cyclic quiver  $\tri$ introduced in \cite[(2.1.3.2)]{DDF} (see also \cite{X97}). It was proved in \cite[2.5.3]{DDF} that $\dbfHa$ is isomorphic to the quantum loop algebra $\bfU(\afgl)$. According to \cite[2.6.1, 2.6.3(5) and 3.9.2]{DDF} we have the following result.
\begin{Lem} \label{presentation-dbfHa}
The algebra $\dbfHa$ is the algebra over $\mbq(\up)$ generated by
$u_A^+$, $K_{i}^{\pm 1}$, $u_A^-$ $(A\in\afThnp,\,i\in I)$ subject to
the following relations:
\begin{itemize}
\item[(1)]
$K_iK_j=K_jK_i$, $K_iK_i^{-1}=1$, $u_0^+=u_0^-=1$;
\item[(2)]
$K\su{\bfj} u_A^+=\up^{\lr{\bfd(A),\bfj}}u_A^+K\su\bfj$,
$u_A^-K\su\bfj=\up^{\lr{\bfd(A),\bfj}}K\su\bfj u_A^-$, where
$K\su\bfj=K_1^{j_1}\cdots K_n^{j_n}$ for $\bfj\in\afmbzn$;
\item[(3)]
$u_A^+u_B^+=\sum_{C\in\afThnp}\up^{\lan \bfd(A),\bfd(B)\ran}\vi_{A,B}^C u_C^+$;
\item[(4)]
$u_A^-u_B^-=\sum_{C\in\afThnp}\up^{\lan \bfd(B),\bfd(A)\ran}\vi_{B,A}^C u_C^-$;
\item[(5)] {\rm commutator relations}:  for all $\la,\mu\in\afmbnn$,
\begin{equation*}
\aligned
\up^{\lan\mu,\mu\ran}&\sum_{\al,\bt\in\afmbnn\atop\la-\al=\mu-\bt\geq 0}\vi_{\la,\mu}^{\al,\bt}
\up^{\lan \bt,\la+\mu-\bt\ran}\ti K\su{\mu-\bt}u_{A_\bt}^-u_{A_\al}^+
=\up^{\lan\mu,\la\ran}\sum_{\al,\bt\in\afmbnn\atop\la-\al=\mu-\bt\geq 0}
{\vi_{\la,\mu}^{\al,\bt}}\up^{\lan \mu-\bt,\al\ran+\lan \mu,\bt\ran}
\ti K\su{\bt-\mu}u_{A_\al}^+u_{A_\bt}^-,\endaligned
\end{equation*}
\end{itemize}
where $\ti K\su\nu :=(\ti K_1)^{\nu_1}\cdots(\ti K_n)^{\nu_n}$ with $\ti K_i=K_iK_{i+1}^{-1}$ for $\nu\in\afmbzn$, and
\begin{equation*}
\vi_{\la,\mu}^{\al,\bt}=\up^{2\sum_{1\leq i\leq n}(\la_i-\al_i)(1-\al_i-\bt_i)}\prod_{1\leq i\leq n\atop 0\leq s\leq\la_i-\al_i-1}\frac{1}{\up^{2(\la_i-\al_i)}-\up^{2s}}.
\end{equation*}
\end{Lem}

Note that the set $\{u_A^+K^{\bfj}u_B^-\mid A,B\in\afThnp,\,\bfj\in\afmbzn\}$ forms a $\mbq(\up)$-basis of $\dbfHa$.
\subsection{}
We now recall the definition of affine quantum Schur algebras following   \cite{Lu99}.
Let $\field$ be a field and fix an $\field[\ep,\ep^{-1}]$-free
module $V$ of rank $r\in\mbn$, where $\ep$ is an indeterminate. A
lattice in $V$ is, by definition, a free $\field[\ep]$-submodule
$L$ of $V$ satisfying
$V=L\otimes_{\field[\ep]}\field[\ep,\ep^{-1}]$.
Let $\afFn={\mathscr F}_{\vtg,n}$ be the set of all  filtrations  $\bfL=(L_i)_{i\in\mbz}$ of lattices, where each $L_i$ is a lattice in $V$ such that
$L_{i-1}\han L_i$ and $L_{i-n}=\ep L_i$, for all $i\in\mbz$. The
group $G$ of automorphisms of the $\field[\ep,\ep^{-1}]$-module $V$
acts on $\afFn$ by $g\cdot\bfL=(g(L_i))_{i\in\mbz}$ for $g\in G$ and
$\bfL\in\afFn$.  The group $G$ acts on $\afFn\times\afFn$ by
$g\cdot(\bfL,\bfL')=(g\cdot\bfL,g\cdot\bfL')$.

Recall the set $\afThnr$ given in \S 1.
By \cite[1.5]{Lu99}
there is a bijection between the set of $G$-orbits in
$\afFn\times\afFn$ and $\afThnr$ by sending $(\bfL,\bfL')$
to $A=(a_{i,j})_{ij\in\mbz}$, where
$a_{i,j}=\dim_\field {L_i\cap L_j'}/({L_{i-1}\cap L_j'+L_i\cap L_{j-1}'}).$
Let $\sO_A\han\afFn\times\afFn$ be
the $G$-orbit corresponding to the matrix $A\in\afThnr$.

Let $\field=\field_q$ be the finite field of $q$
elements. For $A,A',A''\in\afThnr$ and $(\bfL,\bfL'')\in\sO_{A''}$ let
$\nu_{A,A',A'';q}=\#\{\bfL'\in\afFn  \mid (\bfL,\bfL')
\in\sO_{A},(\bfL',\bfL'')\in\sO_{A'}\}$.
By \cite[1.8]{Lu99}, there exists a polynomial
$\nu_{A,A',A''}\in\sZ$ in $\up^2$ such that
$\nu_{A,A',A''}|_{\up^2=q}=\nu_{A,A',A'';q}$
for any $q$, a power of a prime number.

Let $\afSr$ be the the free $\sZ$-module with
basis $\{e_{A}\mid A\in\afThnr\}$. According to \cite[1.9]{Lu99} there is a unique associative $\sZ$-algebra structure on $\afSr$ with multiplication
$e_{A}e_{{A'}}=\sum_{A''\in\afThnr}\nu_{A,A',A''}e_{{A''}}$.
Let $\afbfSr=\afSr\ot\mbq(\up)$. The algebras $\afSr$ and $\afbfSr$ are called affine quantum Schur algebras.

For $A\in\afThnr$ let
\begin{equation}\label{dA}
[A]=\up^{-d_A}e_A,\quad\text{ where}\quad d_{A}=\sum_{1\leq i\leq n,\, i\geq k,j<l}a_{i,j}a_{k,l}
\end{equation}
According to \cite[1.11]{Lu99}, the $\sZ$-linear map
\begin{equation}\label{taur}
\tau_r:\afSr\lra\afSr,\;\;[A]\lm [\tA]
\end{equation} is an algebra
anti-involution, where $\tA$ is the transpose of $A$.

\subsection{}
Let $\afsygr$ be the group consisting of all permutations
$w:\mbz\ra\mbz$ such that $w(i+r)=w(i)+r$ for $i\in\mbz$.
The extended affine Hecke algebra $\afHr$ of affine type $A$ over $\sZ$  is the (unital) $\sZ$-algebra with basis
$\{T_w\}_{w\in\affSr}$, and multiplication defined by
\begin{equation*}
\begin{cases} T_{s_i}^2=(\up^2-1)T_{s_i}+\up^2,\quad&\text{for }1\leq i\leq r\\
T_{w}T_{w'}=T_{ww'},\quad&\text{if}\
\ell(ww')=\ell(w)+\ell(w'),
\end{cases}
\end{equation*}
where $s_i\in\affSr$ is defined by
setting
$s_i(j)=j$ for $j\not\equiv i,i+1\mnmod r$, $s_i(j)=j-1$ for
$j\equiv i+1\mnmod r$ and $s_i(j)=j+1$ for $j\equiv i\mnmod r$, and $\ell(w)$ is the length of $w$.

Recall the set $\afLanr$ given in \S 1. Let $\fS_r$ be the subgroup of $\affSr$ generated by $s_i$ for $1\leq i\leq r-1$, which is isomorphic to the symmetric group of degree $r$.
For $\la\in\afLanr$, let $\fS_\la:=\fS_{(\la_1,\ldots,\la_n)}$
be the corresponding standard Young subgroup of $\fS_r$ and let
$x_\la=\sum_{w\in\fS_\la}T_w\in\afHr$.
For $\la,\mu\in\afLanr$, let
$\afmsD_\la=\{d\mid d\in\affSr,\ell(wd)=\ell(w)+\ell(d)\text{ for
$w\in\fS_\la$}\}$ and $\afmsD_{\la,\mu}=\afmsD_{\la}\cap{\afmsD_{\mu}}^{-1}$.
For $\la,\mu\in\afLanr$ and $d\in\afmsD_{\la,\mu}$
define
$\phi_{\la,\mu}^d\in\End_{\afHr}\bigl
(\bop_{\la\in\La_\vtg(n,r)}x_\la\afHr\bigr)$ by
\begin{equation*}
\phi_{\la,\mu}^d(x_\nu h)=\dt_{\mu\nu}\sum_{w\in\fS_\la
d\fS_\mu}T_wh
\end{equation*}
for $\nu\in\afLanr$ and $h\in\afHr$.

For $\la\in\afLanr$,
$1\leq i\leq n$ and $k\in\mbz$ let
\begin{equation}\label{Rtla}
R_{i+kn}^{\la}=\{\la_{k,i-1}+1,\la_{k,i-1}+2,\ldots,\la_{k,i-1}+\la_i
=\la_{k,i}\},
\end{equation}
where $\la_{k,i-1}=kr+\sum_{1\leq t\leq i-1}\la_t.$
By Varagnolo--Vasserot \cite[7.4]{VV99} (see also \cite[9.2]{DF10}), there is
a bijective map
\begin{equation}\label{jmath}
{\jmath_\vtg}:\{(\la, d,\mu)\mid
d\in\afmsD_{\la,\mu},\la,\mu\in\afLanr\}\lra\afThnr
 \end{equation}
sending $(\la, d,\mu)$ to the matrix $A=(|R_k^\la\cap dR_l^\mu|)_{k,l\in\mbz}$.
Varagnolo--Vasserot showed in \cite{VV99} that there is an algebra isomorphism $${\mathfrak h}:\End_{\afHr}\biggl
(\bop_{\la\in\La_\vtg(n,r)}x_\la\afHr\biggr)\ra\afSr$$
such that ${\mathfrak h}(\phi_{\la,\mu}^d)=e_A$, where $A=\jmath_\vtg(\la,d,\mu)$. We identify $\End_{\afHr}\bigl
(\bop_{\la\in\La_\vtg(n,r)}x_\la\afHr\bigr)$ with $\afSr$ via ${\mathfrak h}$.

\subsection{}
It was shown in \cite{DDF} that the double Ringel--Hall algebra $\dbfHa$ and the affine quantum Schur algebra $\afbfSr$ are related by a surjective algebra homomorphism $\zr$.
Let $\afThnpm:=\{A\in\afThn\mid a_{i,j}=0\text{ for }i= j\}$.
For $A\in\afThnpm$ and $\bfj\in\afmbzn$, define $A(\bfj,r)\in
\afbfSr$ by
\begin{equation*}
A(\bfj,r)=\begin{cases}
\sum_{\la\in\La_\vtg(n,r-\sg(A))}\up^{\la\cdot\bfj}[A+\diag(\la)],&\text{ if }\sg(A)\leq r;\\
0,&\text{ otherwise,}\end{cases}
\end{equation*}
where $\la\cdot\bfj=\sum_{1\leq i\leq n}\la_ij_i$.
For $A\in\afThnp$ let
$$\ti u_A^\pm=\up^{\dim \End(M(A))-\dim M(A)}u_A^\pm.$$
We have the following result.
\begin{Thm}[{\cite[3.6.3, 3.8.1]{DDF}}]\label{zr}
For $r\geq 0$, the linear map $\zr:\dbfHa\ra \afbfSr$ satisfying
$$\zr(K^\bfj)=0(\bfj,r),\;\zr(\ti u_A^+)=A(\bfl,r),\;\;\text{and}\;\;
\zr(\ti u_A^-)=(\tA)(\bfl,r),$$
for all $\bfj\in \afmbzn$ and $A\in \afThnp$, is a surjective algebra homomorphism.
\end{Thm}

\section{Canonical bases for affine quantum Schur algebras}

\subsection{}
Let $W_r$ be the subgroup of $\afsygr$ generated by $s_i$ for $1\leq i\leq r$.
For $i,j\in\mbz$ such that $i\not\equiv j\mnmod r$, define $(i,j)\in\affSr$  by
setting
$(i,j)(k)=k$ for $k\not\equiv i,j\mnmod r$, $(i,j)(k)=j+k-i$ for
$k\equiv i\mnmod r$ and $(i,j)(k)=i+k-j$ for $k\equiv j\mnmod r$.
Note that $(i,j)\in W_r$ for all $i,j$.
By definition we have $(i,j)=(i+tr,j+tr)$ for $t\in\mbz$ and $(i,i+1)=s_{i}$.
Let $$T=
\bin_{w\in\Wr,1\leq i\leq r}ws_iw^{-1}
=\{(i,j)\in\Wr|1\leq i\leq r,\,i,j\in\mbz,\ i<j,\ i\not\equiv j\mnmod r\}.$$
For $y,w\in\Wr$, we write $y\leq w$ if there exist $t_i\in T$ ($1\leq i\leq m$) for some $m\in\mbn$ such that $w=t_1t_2\cdots t_my$ and $\ell(t_it_{i+1}\cdots t_my)>\ell(t_{i+1}t_{i+2}\cdots t_my)$ for $1\leq i\leq m$. The partial ordering $\leq$ on $\Wr$ is called the Bruhat order.
Let $\rho$ be the permutation of $\mbz$ sending $j$ to $j+1$ for all $j\in\mbz$. Then we have $\affSr=\langle\rho\rangle\ltimes W_r$, where
$\lan\rho\ran\cong\mbz$ is the subgroup of $\affSr$ generated by $\rho$.
The Bruhat order on $\Wr$ can be extended to $\affSr$ by define $\rho^iy\leq\rho^jw$ (for $y,w\in\Wr$) if and only if $i=j$ and $y\leq w$.

Let $\bar\ :\afHr\ra\afHr$ be the ring involution
defined by $\bar v=v^{-1}$ and $\bar T_w=T_{w^{-1}}^{-1}$.
Let $\sH(W_r)$ be the $\sZ$-subalgebra of $\afHr$ generated by $T_{s_i}$ for $1\leq i\leq r$.
Let $\{C_w'\mid w\in W_r\}$ be the Kazhdan--Lusztig basis of $\sH(W_r)$ defined in \cite[1.1(c)]{KL79}.  For $y,w\in W_r$ and $a,b\in\mbz$ let $P_{\rho^ay,\rho^bw}=\dt_{a,b}P_{y,w}$, where $P_{y,w}\in\sZ$ is the Kazhdan--Lusztig polynomial.
For $w=\rho^ax\in\affSr$ with $a\in\mbz$ and $x\in W_r$, let $C'_w=T_\rho^aC_x'.$
Then for $w\in\affSr$ we have $\ol{C'_w}=C'_w$ and
$$C'_w=\sum_{y\leq w,\,y\in\affSr}\up^{\ell(y)-\ell(w)}P_{y,w}\ti T_y$$
where $\ti T_y=\up^{-\ell(y)}T_y$. The set $\{C'_w\mid w\in\affSr\}$ is
called the canonical basis of $\afHr$.

For $d\in\afmsD_{\la,\mu}$ let $T_{\fS_\la d\fS_\mu}=\sum_{w\in\fS_\la d\fS_\mu}T_w$ and
$\ti T_{\fS_\la d\fS_\mu}=v^{-\ell(d^+)}T_{\fS_\la d\fS_\mu}$, where $d^+$ is the unique longest element in $\fS_\la d\fS_\mu$. According to \cite[(1.10)]{Curtis} and \cite[4.35]{DDPW} we have the following result.
\begin{Lem}\label{C'd +}
For $\la,\mu\in\afLanr$ and $d\in\afmsD_{\la,\mu}$
we have
$$C'_{d^+}=\sum_{y\in\afmsD_{\la,\mu}\atop y\leq d}v^{\ell(y^+)-\ell(d^+)}P_{y^+,d^+}\ti T_{\fS_\la y\fS_\mu},$$
where $y^+$ is the unique longest element in $\fS_\la y\fS_\mu$.
\end{Lem}

\subsection{}
We now recall the definition of canonical bases of affine quantum Schur algebras.
Note that $C_{w_{0,\la}}'=v^{-\ell(w_{0,\la})}x_\la$,
where $w_{0,\la}$ is the longest element in $\fS_\la$.
We define a map $\bar\ :\afSr\ra\afSr$ by $\up\mapsto\bar\up=\up^{-1}$, $f\mapsto\bar f$, where for $f\in\Hom_{\afHr}(x_{\mu}\afHr,x_\la\afHr)$,
$\bar f\in\Hom_{\afHr}(x_{\mu}\afHr,x_\la\afHr)$ is defined by $\bar f(C_{w_0,\mu}'h)=\ol{f(C_{w_0,\mu}')}h$ for $h\in\afHr$. Then the map $\bar\ :\afSr\ra\afSr$ is a ring involution (cf. \cite{Du92}).

For $A\in\afThn$ let
$\ro(A)=\bigl(\sum_{j\in\mbz}a_{i,j}\bigr)_{i\in\mbz}$ and
$\co(A)=\bigl(\sum_{i\in\mbz}a_{i,j}\bigr)_{j\in\mbz}.$
For $A\in\aftiThn$ and $i\not=j\in\mbz$, let\vspace{-2ex}
$$\sg_{i,j}(A)=\begin{cases}\sum\limits_{s\leq i,t\geq j}a_{s,t},\;&\text{ if $i<j$};\\
\sum\limits_{s\geq i,t\leq j}a_{s,t},\; & \text{ if
$i>j$}.\end{cases} $$
 For $A,B\in\aftiThn$, define
$B\pr A$ by the condition $\sg_{i,j}(B)\leq\sg_{i,j}(A)$ \text{ for all } $i\not=j.$
Put $B\p A$ if $B\pr A$ and, for some pair $(i,j)$ with $i\not=j$,
$\sg_{i,j}(B)<\sg_{i,j}(A)$.
For $A,B\in\aftiThn$ define
$B\sqsubseteq A$ \text{ if and only if $B\pr A$, $\co(B)=\co(A)$ and $\ro(B)=\ro(A)$.}
Put $B\sqsubset A$ if $B\sqsubseteq A$ and $B\not=A$.
According to \cite[6.1]{DF10} we know that the order relation $\sqsubseteq$ is a partial order relation on $\aftiThn$.

Lusztig proved in \cite{Lu99} that there is a unique $\sZ$-basis
\begin{equation}\label{bfBr}
\bfBr:=\{\th_{\Ar}\mid A\in\afThnr\}
\end{equation}
 for $\afSr$ such that $\ol{\th_{\Ar}}=\th_{\Ar}$ and
\begin{equation}\label{th_Ar}
\th_{A,r}-[A]\in\sum_{B\in\afThnr\atop B\sqsubset A}\up^{-1}\mbz[\up^{-1}][B],
\end{equation}
(see also \cite[7.6]{DF14}).
The set $\bfBr$ is called the canonical basis of $\afSr$.
\subsection{}

For $w\in\affSr$  let
$\msL(w)=\{(i,j)\in\mbz^2\mid1\leq i\leq r,\ i<j,\ w(i)>w(j)\}$ and $\msR(w)=\{(i,j)\in\mbz^2\mid1\leq j\leq r,\ i<j,\ w(i)>w(j)\}$.
The following result is given in \cite[(3.2.1.1)]{DDF} (see also \cite[5.2]{DF13}).
\begin{Lem}\label{combinatorial description of length}
For $w\in\affSr$, we have $\ell(w)=|\msL(w)|=|\msR(w)|$.
\end{Lem}
For $i\in\mbz$ the image of $i$ in $\mbz/r\mbz$ will be denoted by $\bar i$. The following corollary can be proved by a standard argument by  using  Lemma \ref{combinatorial description of length}. So we omit the proof.

\begin{Coro}\label{(i,j)y}
Let $x\in\affSr$ and $i_0,j_0\in\mbz$ such that $i_0<j_0$ and $\ol{i_0}\not=\ol{j_0}$. Then we have
$x<(i_0,j_0)x$ if and only if $x^{-1}(i_0)<x^{-1}(j_0)$, i.e. $i_0$ occurs in the left of $j_0$ in the sequence $(x(s))_{s\in\mbz}$.
\end{Coro}

For $i\in\mbz$ let $(-\iy,i]=\{\bfa=(a_s)_{s\leq i}|a_s\in\mbz\}$ and $[i,+\iy)=\{\bfa=(a_s)_{s\geq i}|a_s\in\mbz\}$. If either $\bfa,\bfb\in (-\iy,i]$ or $\bfa,\bfb\in [i,+\iy)$, we write $\bfa\leq\bfb$ if $a_s\leq b_s$ for all $s$.
Given $\bfa=(a_s)\in\afmbzn$ and an integer $i$ we let  $(a_s)_{s\leq i}^{\mathrm{sorted}}=(b_s)_{s\leq i}$ such that
$\{a_s|s\leq i\}=\{b_s|s\leq i\}$ and $b_{s-1}\leq b_s$ for $s\leq i$. Similarly we may define $(a_s)_{s\geq i}^{\mathrm{sorted}}$
for $\bfa\in\afmbzn$ and $i\in\mbz$.

By Corollay \ref{(i,j)y} we have the following result.
\begin{Coro}\label{Bruhat order}
Let $y,w\in\affSr$. If $y\leq w$ then for any $i\in\mbz$ we have $(y(s))_{s\leq i}^{\mathrm{sorted}}\leq (w(s))_{s\leq i}^{\mathrm{sorted}}$ and $(y(s))_{s\geq i}^{\mathrm{sorted}}\geq (w(s))_{s\geq i}^{\mathrm{sorted}}$.
\end{Coro}

\subsection{}
Recall the map $\jmath_\vtg$ defined in \eqref{jmath}.
Given $A\in\afThnr$, write $y_A=w$ if  $A=\jmath_\vtg(\la,w,\mu)$. For $A,B\in\afThnr$, define
$B\Boleq A$ by the condition $\ro(B)=\ro(A),\ \co(B)=\co(A)$ and $y_B\leq y_A.$ Put $B\Bol A$ if $B\Boleq A$ and $B\not=A$. Then $\Boleq$ is a partial order relation on $\afThnr$.

Recall that $V$ is a $\field[\ep,\ep^{-1}]$-free
module of rank $r\in\mbn$. Let $\{v_1,v_2,\cdots,v_r\}$ be a fixed $\field[\ep,\ep^{-1}]$-basis of $V$. We set $v_{i+kr}=\ep^{-k}v_i$ for $1\leq i\leq r$ and $k\in\mbz$.
\begin{Lem}\label{L(A)}
Let $A\in\afThnr$, $\la=\ro(A)$ and $\mu=\co(A)$. Let $\bfL(A)=(L_i)_{i\in\mbz}$ and $\bfL'(A)=(L_i')_{i\in\mbz}$ where
\begin{equation*}
\begin{split}
L_{i+kn}&=\spann_\field\bigg\{v_a\big|a\in\bin_{t\leq i+kn}R_t^\la\bigg\}=\spann_\field\bigg\{v_a\big|a\leq\sum_{1\leq j\leq i}\la_j+kr\bigg\}\\
L_{i+kn}'&=\spann_\field\bigg\{v_{y_A(a)}\big|a\in\bin_{t\leq i+kn}R_t^\mu\bigg\}=\spann_\field\bigg\{v_{y_A(a)}\big|a\leq\sum_{1\leq j\leq i}\mu_j+kr\bigg\}
\end{split}
\end{equation*}
for $1\leq i\leq n$ and $k\in\mbz$. Then
we have $(\bfL(A),\bfL'(A))\in\sO_A$.
\end{Lem}
\begin{proof}
By definition we have
$L_i\cap L_j'=\spann_\field\{v_a|a\in\bin_{t\leq i}R_t^\la,\ a\in\bin_{t\leq j}y_A(R_t^\mu)\}$ for $i,j\in\mbz$. Hence for $i,j\in\mbz$
we have
$$\frac{L_i\cap L_j'}{L_{i-1}\cap L_j'+L_i\cap L_{j-1}'}=\spann_\field\{\ol{v}_a|a\in R_i^\la\cap y_A(R_j^\mu)\}.$$
The assertion follows.
\end{proof}

\begin{Lem}\label{Bruhat order and pr}
$(1)$ If $A,B\in\afThnr$ and $B\Boleq A$ then $B\sqsubseteq A$.

$(2)$ If $A,B\in\afThnr$ and $B\Bol A$ then $B\sqsubset A$.
\end{Lem}
\begin{proof}If $B\Boleq A$ then $\ro(B)=\ro(A)$, $\co(B)=\co(A)$ and $y{_B}\leq y{_A}$. We denote $\la=\ro(B)$ and $\mu=\co(B)$. Let $\bfL=\bfL(A)=\bfL(B)$, $\bfL'=\bfL'(A)$ and $\bfL''=\bfL'(B)$. Then by Lemma \ref{L(A)} we have
$(\bfL,\bfL')\in\sO_A$ and $(\bfL',\bfL'')\in\sO_{B}$. By definition for $i,j\in\mbz$ we have
\begin{equation*}
\begin{split}
L_i/(L_i\cap L_{j-1}')&=\spann_\field\{\ol{v}_{y{_A}(a)}|y{_A}(a)\in\bin_{t\leq i}R_t^\la,\ a\in\bin_{t\geq j}R_t^\mu\},\\
L_i/(L_i\cap L_{j-1}'')&=\spann_\field\{\ol{v}_{y{_B}(a)}|y{_B}(a)\in\bin_{t\leq i}R_t^\la,\ a\in\bin_{t\geq j}R_t^\mu\}\\
L_j'/(L_{i-1}\cap L_j')&=\spann_\field\{\ol{v}_{y{_A}(a)}|a\in\bin_{t\leq j}R_t^\mu,\ y{_A}(a)\in\bin_{t\geq i}R_t^\la\},\\
L_j''/(L_{i-1}\cap L_j'')&=\spann_\field\{\ol{v}_{y{_B}(a)}|a\in\bin_{t\leq j}R_t^\mu,\ y{_B}(a)\in\bin_{t\geq i}R_t^\la\}.\\
\end{split}
\end{equation*}
Since $y{_B}\leq y{_A}$, by Corollary \ref{Bruhat order} we have
$\dim_\field(L_i/(L_i\cap L_{j-1}''))\leq\dim_\field(L_i/(L_i\cap
L_{j-1}'))$ and $\dim_\field(L_j''/(L_{i-1}\cap
L_j''))\leq\dim_\field(L_j'/(L_{i-1}\cap L_j'))$. Hence by \cite[1.6(a)]{Lu99} we conclude that
$B\sqsubseteq A$. Thus (1) holds.
Now we assume $B\Bol A$. Suppose that $B\not\p A$.
Then by (1) we have $B\pr A$ and $\ro(B)=\ro(A)$. Hence by
\cite[6.1]{DF10} we see that $B$ and $A$ have the same off
diagonal entries. Since $\ro(B)=\ro(A)$ we must have $A=B$. This is
a contradiction. Hence $B\p A$. The assertion (2) follows.
\end{proof}

\subsection{}
For $A\in\afThnr$ let $y_A^+$ be the  unique longest element in $\fS_\la y_A\fS_\mu$, where $\la=\ro(A)$ and $\mu=\co(A)$.
The following result is given in \cite[7.1]{DF14}.
\begin{Lem}\label{longest element}
For $A\in\afThnr$ we have $\ell(y_A^+)=d_A+\ell(w_{0,\mu})$ where $\mu=\co(A)$ and $d_A$ is given in \eqref{dA}.
\end{Lem}

For $\la,\mu\in\afLanr$ and
$d\in\afmsD_{\la,\mu}$, define
$\th_{\la,\mu}^d\in\afSr$ as follows:
\begin{equation*}
\th_{\la,\mu}^d(x_\nu h)=\dt_{\mu\nu}v^{\ell(w_{0,\mu})}C'_{d^+}h,
\end{equation*}
where $\nu\in\afLanr$, $h\in\afHr$ and $d^+$ is the  unique longest element in $\fS_\la d\fS_\mu$.
\begin{Prop}\label{KL-basis}
Assume $\la,\mu\in\afLanr$, $d\in\msD_{\la,\mu}$ and $A=\jmath_\vtg(\la,d,\mu)\in\afThnr$. Then we have
$$\th_{\la,\mu}^d=\th_{\Ar}=\sum_{B\in\afThnr\atop
B\Boleq A}v^{\ell(y_B^+)-\ell(y_A^+)}P_{y_B^+,y_A^+}[B],$$
where $P_{y_B^+,y_A^+}$ is the Kazhdan--Lusztig polynomial.
\end{Prop}
\begin{proof}
By definition we have $\ol{\th_{\la,\mu}^d}=\th_{\la,\mu}^d$. Furthermore, by Lemma \ref{C'd +} we
conclude that
\begin{equation}\label{thlamud}
\th_{\la,\mu}^d=\sum_{x\in\afmsD_{\la,\mu}\atop x\leq d}v^{\ell(x^+)-\ell(d^+)}P_{x^+,d^+}\ti{\phi_{\la,\mu}^x},
\end{equation}
where $\ti{\phi_{\la,\mu}^x}=v^{\ell(w_{0,\mu})-\ell(x^+)}\phi_{\la,\mu}^x$.
In addition, by Lemma \ref{longest element} we have $[B]=\ti{\phi_{\la,\mu}^x}$ for $B\in\afThnr$ with $B=\jmath_\vtg(\la,x,\mu)$. Consequently, by Lemma \ref{Bruhat order and pr}
and the uniqueness of $\th_{A,r}$ we conclude that $\th_{A,r}=\th_{\la,\mu}^d$. The assertion follows.
\end{proof}

\section{Connection between $\bfBr$ and $\bfBNap$}

\subsection{}
Let $\bfU(\afsl)$ be the $\mbq(\up)$-subalgebra of $\dbfHa$ generated by the elements $u^+_{\afE_{i,i+1}}$, $u_{\afE_{i+1,i}}^-$ and $\ti K_i^{\pm 1}$ for $i\in I$. Then $\bfU(\afsl)$ is isomorphic to quantum affine $\afsl$. Let $\bfU(\afsl)^+$ be the  $\mbq(\up)$-subalgebra of $\bfU(\afsl)$ generated by the elements $u^+_{\afE_{i,i+1}}$ for $i\in I$.
Let $\afUslp$ be the $\sZ$-subalgebra of
$\bfU(\afsl)^+$ generated by $\ti u_{m\afE_{i,i+1}}^+$ for $i\in I$ and $m\in\mbn$.
The algebra $U(\afsl)^+_\sZ$ is the $\sZ$-form of $\bfU(\afsl)^+$.

Let $\dHap=\spann_\sZ\{\tu_A^+\mid A\in\afThnp\}$. Then $\dHap$ is a $\sZ$-subalgebra of $\dbfHa$ and $\afUslp$ is a proper subalgebra of $\dHap$.
According to \cite[Prop 7.5]{VV99}, there is a unique $\sZ$-basis
\begin{equation}\label{bfBn}
\bfBn:=\{\th_A^+\mid A\in\afThnp\}
\end{equation}
for $\dHap$ such that $\ol{\th_A^+}=\th_A^+$ and
\begin{equation}\label{aaa}
\th_A^+-\ti u_A^+\in\sum_{B\p A,\,B\in\afThnp\atop
\bfd(B)=\bfd(A)}\up^{-1}\mbz[\up^{-1}]\ti u_B^+.
\end{equation}
The set $\bfBn$ is called the canonical basis of $\dHap$. For $A,B\in\afThnp$ we write
\begin{equation}\label{fABC}
\th_A^+\th_B^+=\sum_{C\in\afThnp}
\sff_{A,B,C}\th_C^+,
\end{equation}
where $\sff_{A,B,C}\in\sZ$. Note that if $\sff_{A,B,C}\not=0$ then
$\bfd(C)=\bfd(A)+\bfd(B)$.

A matrix $A=(a_{i,j})\in\afThn$ is said to be aperiodic if for every integer $l\neq0$ there exists $1\leq i\leq n$ such that $a_{i,i+l}=0$. Let $\afThnap$ be the set of all aperiodic matrices in $\afThn$. Let
$\afThnpap=\afThnp\cap\afThnap$.

By Lusztig \cite{Lu92} we know that the set
\begin{equation}\label{bfBnap}
\bfBn^{\text{ap}}:=\{\th_A^+\mid A\in\afThnpap\}
\end{equation}
forms a $\sZ$-basis for $\afUslp$ and is called the canonical basis of $\afUslp$.
The following positivity result for $\afUslp$ was proved by Lusztig.
\begin{Thm}[{\cite[14.4.13]{Lubk}}]\label{positive affine sln}
For $A,B,C\in\afThnpap$ we have $\sff_{A,B,C}\in\mbn[\up,\up^{-1}]$.
\end{Thm}

\subsection{}
Let
$\dHaz$ be the $\sZ$-subalgebra of $ \dbfHa$
generated by $K_i^{\pm1}$ and $\leb{K_i;0\atop t}\rib$ for $1\leq i\leq n$ and $t>0$, where
$\big[ {K_i;0 \atop t} \big] =\prod_{s=1}^t \frac
{K_i\up^{-s+1}-K_i^{-1}\up^{s-1}}{\up^s-\up^{-s}}.$
Let $\dHapz=\dHap\dHaz$. Then $\dHapz$ is a $\sZ$-subalgebra of
$\dHa$.

Recall the map $\zr$ defined in Theorem \ref{zr}.
Let $\afSrpz$ be the $\sZ$-submodule of $\afSr$ spanned by the elements
$A(\bfl,r)[\diag(\la)]$ for $A\in\afThnp$ and $\la\in\afLa(n,r)$. Since $\afSrpz=\zr(\dHapz)$, we conclude that $\afSrpz$ is a $\sZ$-subalgebra of $\afSr$. The algebra $\afSrpz$ is called a Borel subalgebra of $\afSr$.

\begin{Lem}
The set $\{\th_{A+\diag(\la),r}\mid A\in\afThnp,\,\la\in\afLa(n,r-\sg(A))\}$
forms a $\sZ$-basis of $\afSrpz$.
\end{Lem}
\begin{proof}
By definition the set $\{[A+\diag(\la)]\mid|A\in\afThnp,\,\la\in\afLa(n,r-\sg(A))\}$ forms a $\sZ$-basis of $\afSrpz$. Furthermore, by \eqref{th_Ar}, for $A\in\afThnp$ and $\la\in\afLa(n,r-\sg(A))$,
we have
\begin{equation*}
\th_{A+\diag(\la),r}-[A+\diag(\la)]\in
\sum_{B\in\afThnp,\,\mu\in\afLa(n,r-\sg(B))
\atop B+\diag(\mu)\sqsubset A+\diag(\la)}\sZ[B+\diag(\mu)].
\end{equation*}
The assertion follows.
\end{proof}

According to \cite[7.7(2) and 7.9]{DF14} we have the following result (see also \cite[3.7]{Fu14}).
\begin{Lem}\label{zr(thA+)}
For $A\in\afThnp$ we have $\zr(\th_A^+)=\sum_{\mu\in\afLa(n,r-\sg(A))}\th_{A+\diag(\mu),r}$. In particular we have
\begin{equation*}
[\diag(\la)]\zr(\th_A^+)=
\begin{cases}
\th_{A+\diag(\la-\ro(A)),r}&\text{if $\la-\ro(A)\in\afmbnn$}\\
0 &\text{otherwise.}
\end{cases}
\end{equation*}
and
\begin{equation*}
\zr(\th_A^+)[\diag(\la)]=
\begin{cases}
\th_{A+\diag(\la-\co(A)),r}&\text{if $\la-\co(A)\in\afmbnn$}\\
0 &\text{otherwise.}
\end{cases}
 \end{equation*}
for $\la\in\afLanr$.
\end{Lem}

For $A,B\in\afThnr$
we write
\begin{equation}\label{gABCr}
\th_{A,r}\th_{B,r}=\sum_{C\in\afThnr}\g_{A,B,C,r}\th_{C,r}
\end{equation}
where $\g_{A,B,C,r}\in\sZ$. If $\g_{A,B,C,r}\not=0$ then we have
$\co(A)=\ro(B)$, $\ro(A)=\ro(C)$ and $\co(B)=\co(C)$.

\begin{Lem}\label{ke3}
Let $A,B\in\afThnp$, $\la\in\afLa(n,r-\sg(A))$ and $\mu\in\afLa(n,r-\sg(B))$. If $\co(A)+\la=\ro(B)+\mu$ then
we have $$\g_{A+\diag(\la),B+\diag(\mu),C',r}=
\begin{cases}
\sff_{A,B,C}&\text{ if $C'=C+\diag(\la+\ro(A-C))$ for some $C\in\afThnp$},\\
0&\text{ otherwise.}
\end{cases}
$$
for $C'\in\afThnr$, where $\sff_{A,B,C}$ is as given in \eqref{fABC}.
\end{Lem}
\begin{proof}
By Lemma
\ref{zr(thA+)}  we have
\begin{equation*}
\begin{split}
\th_{A+\diag(\la),r}\th_{B+\diag(\mu),r}
&=[\diag(\la+\ro(A))]\zr(\th_A^+)\zr(\th_B^+)[\diag(\mu+\co(B))]\\
&=\sum_{C\in\afThnp,\,\bfd(C)=\bfd(A)+\bfd(B)
\atop \la+\ro(A)-\ro(C)\in\afmbnn
}\sff_{A,B,C}\th_{C+\diag(\la+\ro(A)-\ro(C)),r}
[\diag(\mu+\co(B))].
\end{split}
\end{equation*}
If $\bfd(C)=\bfd(A)+\bfd(B)$ then we have
$\ro(C)-\co(C)=\ro(A+B)-\co(A+B)$ and hence
$\co(C)+\la+\ro(A)-\ro(C)=\la+\co(A+B)-\ro(B)=\mu+\co(B)$.
Thus we have
\begin{equation*}
\th_{A+\diag(\la),r}\th_{B+\diag(\mu),r}
=\sum_{C\in\afThnp,\,\bfd(C)=\bfd(A)+\bfd(B)
\atop \la+\ro(A)-\ro(C)\in\afmbnn
}\sff_{A,B,C}\th_{C+\diag(\la+\ro(A)-\ro(C)),r}.
\end{equation*}
The assertion follows.
\end{proof}

\subsection{}
For $m\in\mbz$ there is a map
\begin{equation}\label{etam}
\etam:\afThn\ra\afThn
\end{equation}
defined by sending $A=(a_{i,j})_{i,j\in\mbz}$ to $(a_{i,mn+j})_{i,j\in\mbz}$. Note that if $A=\jmath_\vtg(\la,d,\mu)\in\afThnr$ then $\etam(A)=\jmath_\vtg(\la,d\rho^{mr},\mu)\in\afThnr$.

\begin{Lem}\label{ke1}
Let $A\in\afThn$ and $m\in\mbz$.
If $a_{i,j}=0$ for $1\leq i\leq n$ and $j\leq mn$,
then $\etak(A)\in\afThnp$ for $k\leq m-1$.
\end{Lem}
\begin{proof}
Let $B^{(k)}=\etak{(A)}$. If $k\leq m-1$, $1\leq i\leq n$ and $i\geq j$,
then $kn+j\leq (m-1)n+j\leq (m-1)n+i\leq mn$ and hence $b_{i,j}^{(k)}=a_{i,kn+j}=0$.
Thus $B^{(k)}\in\afThnp$ for $k\leq m-1$.
\end{proof}

\begin{Lem}\label{ke2}
Let $A\in\afThnr$ with $\la=\ro(A)$ and $\mu\in\co(A)$. Then we have
$\th_{A,r}\cdot\th_{\mu,\mu}^{\rho^{mr}}=\th_{\etam(A),r}
=\th_{\la,\la}^{\rho^{mr}}
\cdot\th_{A,r}$ for $m\in\mbz$.
\end{Lem}
\begin{proof}
Note that $C'_{w_{0,\mu}}=\up^{-\ell(w_{0,\mu})}x_\mu$.
Since $\rho^rx=x\rho^r$ for $x\in\affSr$ we have $\fS_\mu\rho^{mr}\fS_\mu=\fS_\mu\fS_\mu\rho^{mr}=\fS_\mu\rho^{mr}$.
It follows that $w_{0,\mu} {\rho^{mr}}$ is the longest element in $\fS_\mu\rho^{mr}\fS_\mu$. This together with Proposition \ref{KL-basis} implies that
\begin{equation}\label{eq1 ke2}
\th_{A,r}\th_{\mu,\mu}^{\rho^{mr}}(C'_{w_{0,\mu}})
=\th_{A,r}(C'_{w_{0,\mu}\cdot{\rho^{mr}}})=
\th_{A,r}(C'_{w_{0,\mu}}T_\rho^{mr})=C'_{d^+\rho^{mr}},
\end{equation}
where $d\in\afmsD_{\la,\mu}$ is such that $\jmath_\vtg(\la,d,\mu)=A$
and $d^+$ is the unique longest element in $\fS_\la d\fS_\mu$. Furthermore since $\fS_\la d\rho^{mr}\fS_\mu=\fS_\la d\fS_\mu\rho^{mr}$, we see that $d^+{\rho^{mr}}$ is the longest element in $\fS_\la d\rho^{mr}\fS_\mu$. It follows from \eqref{eq1 ke2} that
$$\th_{\etam(A),r}(C'_{w_{0,\mu}})=\th_{\la,\mu}^{d\rho^{mr}}
(C'_{w_{0,\mu}})=C'_{d^+{\rho^{mr}}}
=\th_{A,r}\th_{\mu,\mu}^{\rho^{mr}}(C'_{w_{0,\mu}}).$$
Thus we have $\th_{A,r}\cdot\th_{\mu,\mu}^{\rho^{mr}}=\th_{\etam(A),r}$.
This implies that $\th_{\mu,\la}^{d^{-1}}\cdot\th_{\la,\la}^{\rho^{-mr}}=\th_{\mu,\la}^{d^{-1}\rho^{-mr}}$.
Applying the map $\tau_r$ given in \eqref{taur}, we get
$\th_{\la,\la}^{\rho^{mr}}
\cdot\th_{A,r}=\tau_r(\th_{\mu,\la}^{d^{-1}}\cdot\th_{\la,\la}^{\rho^{-mr}})
=\tau_r(\th_{\mu,\la}^{d^{-1}\rho^{-mr}})=\th_{\etam(A),r}$.
\end{proof}

Assume $N\geq n$. There is a natural injective map
$$\ti{\,\,}:\afThn\lra\Th_\vtg(N),\quad A=(a_{i,j})\longmapsto\ti
A=(\ti a_{i,j}),$$ where $\ti A=(\ti a_{i,j})$
is defined by
\begin{equation*}\label{AtoAtilde}
\ti a_{k,l+mN}=\begin{cases} a_{k,l+mn}, &\text{ if }1\leq k,l\leq n;\\
0, &\text{ if either }n< k\leq N\text{ or }n< l\leq N\end{cases}
\end{equation*}
for  $m\in\mbz$. Note that the map $\ti{\,\,}:\afThn\lra\Th_\vtg(N)$ induces a map from $\afThnp$ to $\afTh^+(N)$.
Similarly, there is an injective map
\begin{equation*}\label{tilde-map}
\ti\ :\mbz_\vtg^n\lra\mbz_\vtg^N,\;\;\la\longmapsto\ti\la,
\end{equation*}
 where $\ti\la_i=\la_i$ for $1\leq i\leq n $ and $\ti\la_i=0$ for
$n+1\leq i\leq N$.

It is easy to see that there is an injective
algebra homomorphism (not sending 1 to 1)
$$\iota_{n,N}:\afbfSr\lra\afbfSNr,\;\;[A]\longmapsto [\ti A]\;\;\text{for $A\in\afThnr$}$$
(see \cite[\S 4.1]{DDF}).

Let $\afThnrap=\afThnap\cap\afThnr$.
\begin{Lem}\label{prop iota}
Assume $N>n$. Then for $A\in\afThnr$ we have $\ti A\in\afThNrap$ and
$\iota_{n,N}(\th_{A,r})=\th_{\ti A,r}$. In particular we have
$\g_{A,B,C,r}=\g_{\ti A,\ti B,\ti C,r}$ for $A,B,C\in\afThnr$, where $\g_{A,B,C,r}$ is as given in \eqref{gABCr}.
\end{Lem}
\begin{proof}
The first assertion follows from the definition of $\ti A$.
The second assertion follows from Proposition \ref{KL-basis} and \eqref{thlamud}.
\end{proof}

Recall the map $\eta_m$ defined in \eqref{etam}.
The structure constants for the canonical basis $\bfBr=\{\th_{A,r}\mid A\in\afThnr\}$ of the affine quantum Schur algebra $\afbfSr$ and the canonical basis $\bfBNap=\{\th_A^+\mid A\in\afThNpap\}$ of $\afbfUslNp$ are related as follows.

\begin{Thm}\label{prop of canonical basis for affine q-Schur algebras}
Assume $N\geq n$. Let $A,B,C\in\afThnr$ and $C'\in\afThNr$.

$(1)$ We have
$$\g_{\ti{\etak(A)},\ti{\etak(B)},C',r}=
\begin{cases}
\g_{A,B,X,r}&\quad\text{if $C'=\etakt(X)$ for some $X\in\afThnr$}\\
0&\quad\text{otherwise}
\end{cases}$$
for $k\in\mbz$, where $\g_{A,B,X,r}$ is as given in \eqref{gABCr}.

$(2)$ If $N>n$ and $\co(A)=\ro(B)$, then there exist $k_0\in\mbz$ such that for $k\leq k_0$, $\ti{\etak(A)},\ti{\etak(B)},\ti{\etakt(C)}\in\afThNp\cap\afThNrap$ and
$\g_{A,B,C,r}=\sff_{\ti{\etak(A)},\ti{\etak(B)},\ti{\etakt(C)}}$, where
$\sff_{\ti{\etak(A)},\ti{\etak(B)},\ti{\etakt(C)}}$ is as given in \eqref{fABC}.
\end{Thm}
\begin{proof}
If $\co(A)\not=\ro(B)$ then $\th_{A,r}\th_{B,r}=\th_{\etak(A)}\th_{\etak(B)}=0$ for any $k\in\mbz$. Now we assume
$\co(A)=\ro(B)$.
Let $\la=\ro(A)$ and $\nu=\co(B)$.
Then by Lemma \ref{ke2} we have
\begin{equation}\label{eq1 positive affine Schur}
\th_{\etak(A),r}
\th_{\etak(B),r}=
\th_{\la,\la}^{\rho^{kr}}\th_{A,r}
\th_{B,r}\th_{\nu,\nu}^{\rho^{kr}}=\sum_{X\in\afThnr}\g_{A,B,X,r}\th_{\etak(X),r}
\th_{\nu,\nu}^{\rho^{kr}}=\sum_{X\in\afThnr}\g_{A,B,X,r}\th_{\etakt(X),r}
\end{equation}
for $k\in\mbz$.
Applying $\iota_{n,N}$ to \eqref{eq1 positive affine Schur} gives that
$$\iota_{n,N}(\th_{\etak(A),r})\iota_{n,N}
(\th_{\etak(B),r})=
\sum_{X\in\afThnr}\g_{A,B,X,r}\iota_{n,N}(\th_{{\etakt(X)},r})$$
for $k\in\mbz$.
Thus by Lemma \ref{prop iota} we have
\begin{equation}\label{eq2 positive affine Schur}
\th_{\ti{\etak(A)},r}\th_{\ti{\etak(B)},r}=
\sum_{X\in\afThnr}\g_{A,B,X,r}\th_{\ti{\etakt(X)},r}
\end{equation}
for $k\in\mbz$. The assertion (1) follows.
The assertion (2)
follows from the assertion (1), Lemma \ref{ke3}, Lemma \ref{ke1} and Lemma \ref{prop iota}.
\end{proof}

As a corollary to Theorem \ref{prop of canonical basis for affine q-Schur algebras}, together with
Theorem \ref{positive affine sln} we have the following  positivity property for $\afbfSr$. This gives an alternate approach to Lusztig's result on positivity property for $\afbfSr$ in \cite[4.5]{Lu99}.

\begin{Coro}\label{positive affine Schur}
For $A,B,C\in\afThnr$ we have $\g_{A,B,C,r}\in\mbn[\up,\up^{-1}]$.
\end{Coro}

\subsection{}
There is an injective map from $\bfBn$ to $\bfBNap$ defined by sending $\th_A^+$ to $\th_{\ti A}^+$ for $A\in\afThnp$.
The structure constants for the canonical basis $\bfBn$ of $\dHap$ and the canonical basis $\bfBN^{\text{ap}}$ of $\afUslNp$ are related as follows.
\begin{Thm}\label{relation bfBn bfBNap}
Assume $N> n$. For $A,B,C\in\afThnp$ we have $\sff_{A,B,C}=\sff_{\ti A,\ti B,\ti C}$, where $\sff_{A,B,C}$ is as given in \eqref{fABC}.
\end{Thm}
\begin{proof}
There exist $\la,\mu\in\afmbnn$ such that
$\la+\co(A)=\mu+\ro(B)$ and $\la+\ro(A)-\ro(C)\in\afmbnn$. Let $r=\sg(\la)+\sg(A)$.
Then by Lemma \ref{ke3} and Lemma \ref{prop iota} we have $\sff_{A,B,C}=\g_{A+\diag(\la),B+\diag(\mu),C+\diag(\la+\ro(A-C))}
=\g_{\ti A+\diag(\ti\la),\ti B+\diag(\ti\mu),\ti C+\diag(\ti\la+\ro(\ti A-\ti C))}=\sff_{\ti A,\ti B,\ti C}$.
\end{proof}

The following result is a generalization of Theorem \ref{positive affine sln}, which gives the positivity property for $\dHap$.

\begin{Coro}
For $A,B,C\in\afThnp$ we have $\sff_{A,B,C}\in\mbn[\up,\up^{-1}]$.
\end{Coro}
\begin{proof}
The assertion follows from Theorem \ref{positive affine sln} and Theorem \ref{relation bfBn bfBNap}.
\end{proof}
\section{Positivity properties for $\dbfU$ }
\subsection{}
Recall that $I=\mbz/n\mbz$ and $I$ is identified with $\{1,2,\cdots,n\}$.
There is an algebra grading over $\mbz[I]$
$$\bfU(\afsl)=\bop_{\nu\in\mbz[I]}\bfU(\afsl)_\nu$$
defined by the condition $\bfU(\afsl)_{\nu'}\bfU(\afsl)_{\nu''}\han\bfU(\afsl)_{\nu'+\nu''}$,
$\ti K_i\in\bfU(\afsl)_0$, $u^+_{\afE_{i,i+1}}\in\bfU(\afsl)_{i}$,
$u_{\afE_{i+1,i}}^-\in\bfU(\afsl)_{-i}$ for all $\nu',\nu''\in\mbz[I]$,
$i\in I$.

Let us recall the definition of  the modified quantum affine algebra $\dbfU$ of $\bfU(\afsl)$.
Let $X$ be the quotient of $\afmbzn$ by the subgroup generated by
the element ${\bf 1}$, where ${\bf 1}_i=1$ for all $i$.
For $\la\in\afmbzn$ let $\bar\la\in X$ be the image of $\la$ in $X$. Let $Y=\{\mu\in\afmbzn\mid\sum_{1\leq i\leq n}\mu_i=0\}$. For $\bar\la\in X$ and $\mu\in Y$ we set $ \mu\cdot \bar\la=\sum_{1\leq
i\leq n}\la_i\mu_i$.

For $i\in I$ let  $\afbse_i\in\afmbnn$ be the element
satisfying $(\afbse_i)_j=\dt_{i,j}$ for $j\in I$.
There is a natural map  $I\ra X$ defined by sending $i$ to $\ol{\afal_i}$, where $\afal_i=\afbse_i-\afbse_{i+1}$.
The imbedding $I\ra X$ induce a homomorphism $\iota:\mbz[I]\ra X$.

For $\bar\la,\bar\mu\in X$ we set
$${}_{\bar\la}\bfU(\afsl)_{\bar\mu}=\bfU(\afsl)\bigg/
\bigg(\sum_{\bfj\in Y}(\Kbfj-
 \up^{\bfj\cdot\bar\la})\bfU(\afsl)+\sum_{\bfj\in Y}\bfU(\afsl)(\Kbfj
 -\up^{\bfj\cdot\bar\mu})\bigg).$$
Let $\pi_{\bar\la,\bar\mu}:\bfU(\afsl)\ra
{}_{\bar\la}\bfU(\afsl)_{\bar\mu}$ be the canonical projection.
Let  $$\dbfU:=\bop\limits_{\bar\la,\bar\mu\in X}
{}_{\bar \la}\bfU(\afsl)_{\bar\mu}.$$
We define the product in $\dbfU$ as follows.
Let $\la',\mu',\la'',\mu''\in X$ and $\nu',\nu''\in\mbz[I]$ with $\la'-\mu'=\iota(\nu')$ and $\la''-\mu''=\iota(\nu'')$. For  $t\in\bfU(\afsl)_{\nu'}$,
$s\in\bfU(\afsl)_{\nu''}$,  define
$$\pi_{\la',\mu'}(t)\pi_{\la'',\mu''}(s)=\begin{cases}\pi_{\la',\mu''}(ts),
& \text{if } \mu'=\la''\\
0& \text{otherwise}.
\end{cases}$$
Then $\dbfU$ becomes an associative $\mbq(\up)$-algebra structure with respect to the above product.

\subsection{}
Let $\dbfHasl$ be the subalgebra of $\dbfHa$ generated by the elements $u_A^+$, $u_A^-$ and $\ti K_i^{\pm 1}$ for $A\in\afThnp$ and $i\in I$.
The algebra $\dbfHasl$ is a $\mbz[I]$-graded algebra with
$$\deg(u_A^+)=\sum\limits_{1\leq i\leq n}d_i i,\
\deg(u_A^-)=-
\sum_{1\leq i\leq n}d_i i\ \text{and}\ \deg(\ti K_i^{\pm 1})=0$$
for $A\in\afThnp$ and $1\leq i\leq n$, where $(d_i)_{i\in\mbz}=\bfd(A)$.

Let  $$\ddbfHasl:=\bop\limits_{\bar\la,\bar\mu\in X}
{}_{\bar \la}\dbfHasl_{\bar\mu},$$ where
${}_{\bar\la}\dbfHasl_{\bar\mu}=\dbfHasl\big/
\big(\sum_{\bfj\in Y}(\Kbfj-
 \up^{\bfj\cdot\bar\la})\dbfHasl+\sum_{\bfj\in Y}\dbfHasl(\Kbfj
 -\up^{\bfj\cdot\bar\mu})\big).$
As in the case of $\dbfU$, there is a natural associative $\mbq(\up)$-algebra structure on $\ddbfHasl$ inherited from that of $\dbfHasl$. We will naturally regard $\dbfU$ as a subalgebra of $\ddbfHasl$.

For $\bar\la,\bar\mu\in X$, let $\pi_{\bar\la,\bar\mu}:\dbfHasl\ra
{}_{\bar\la}\dbfHasl_{\bar\mu}$ be the canonical projection.
The algebra $\ddbfHasl$ is naturally a $\dbfHasl$-bimodule defined by
\begin{equation*}\label{bimodule}
t'\pi_{\la',\la''}(s)t''=\pi_{\la'+\iota(\nu'),\la''-\iota(\nu'')}(t'st'') \end{equation*}
for $t'\in\dbfHasl_{\nu'}$, $s\in\dbfHasl$, $t''\in\dbfHasl_{\nu''}$ and $\la',\la''\in X$.

For $\bar\la\in X$ let $1_{\bar\la}=\pi_{\bar\la,\bar\la}(1)$.
The map $\zr$ defined in Theorem \ref{zr} induces an
surjective algebra homomorphism
$$\dzr:\ddbfHasl\ra\afbfSr$$ such that
for $A\in\afThnp$ and
$\bar\la\in X$, $\dzr(u_A^\pm1_{\bar\la})=\zr(u_A^\pm)[\diag(\mu)]$, if  $\bar\la=\bar\mu$ for some $\mu\in\afLanr$, and $\dzr(u_A^\pm1_{\bar\la})=0$ otherwise (cf. \cite[3.6]{Fu13}).

The maps $\dzr$ induce an algebra homomorphism
$$\dz:\ddbfHasl\ra\prod_{r\geq 0}\afbfSr$$
such that $\dz(x)=(\dzr(x))_{r\geq 0}$ for $x\in\ddbfHasl$.
The following result is a generalization of Lusztig \cite[3.5]{Lu00}.

\begin{Thm} \label{injective}
The map $\dz:\ddbfHasl\ra\prod_{r\geq 0}\afbfSr$ is injective.
\end{Thm}
\begin{proof}
Note that  the set
$\{1_{\bar\la}\ti u_A^+\ti u_B^-\mid A,B\in\afThnp,\,\bar\la\in
X\}$ forms a $\mbq(\up)$-basis for
$\ddbfHasl$. We use reduction to
absurdity. Assume $x=\sum_{A\in\afThnpm,\,\bar\la\in
X}\beta_{A,\bar\la}1_{\bar\la}\ti u_{A^+}^+\ti u_{\tAm}^-
\not=0\in\ddbfHasl$ is such that
${\dot\zeta}(x)=0$. Then there exist $\bfa\in
X$ such that $1_{\bfa}x\not=0$. Since the set
$$\sT:=\{A\mid
A\in\afThnpm,\,\beta_{A,\bfa}\not=0\}$$ is finite we may choose a
maximal element $B$ in $\sT$ with respect to $\pr$.  We choose
$\mu\in\afmbnn$ such that $\bar\mu=\bfa$ and $\mu\geq \ro(B)$. Let
$r_0=\sg(\mu)$. Then we have
$$0=[\diag(\mu)]
{\dot\zeta}_{r_0}(x)
=\sum_{A\in\sT}\beta_{A,\bfa}[\diag(\mu)]A^+(\bfl,r)A^-(\bfl,r).
$$
By \cite[3.7.3]{DDF} we have
$$A^+(\bfl,r)A^-(\bfl,r)=A(\bfl,r)+\sum_{C\in\afThnr,\,C\p A}\ga_{A,C}[C]$$
where $\ga_{A,C}\in\mbq(\up)$.
This implies that
\begin{equation*}
\begin{split}
&\qquad\sum_{A\in\sT}\beta_{A,\bfa}[\diag(\mu)]A^+(\bfl,r)A^-(\bfl,r)
\\&=
\bt_{B,\bfa}[\diag(\mu)]\bigg(B(\bfl,r)+\sum_{C\in\afThnr\atop
C\p B}\ga_{B,C}[C]\bigg)
+\sum_{A\in\sT\atop B\not\pr A}\bt_{A,\bfa}[\diag(\mu)]
\bigg(A(\bfl,r)+\sum_{C\in\afThnr\atop
C\p A}\ga_{A,C}[C]\bigg)\\
&=\beta_{B,\bfa}[B+\diag(\mu-\ro(B))]+f
\end{split}
\end{equation*} where
$f$ is a linear combination of $[C'+\diag(\nu)]$ such that
$C'\not= B$, $C'\in\afThnpm$ and $\nu\in\afTh(n,r-\sg(C'))$. Thus we have $\beta_{B,\bfa}=0$.
This is a contradiction.
\end{proof}
\subsection{}

Let $\dbfBn$ be the canonical basis of $\dbfU$ defined in \cite{Lubk}.
Let $\phi_{r+n,r}:\afbfSrn\ra\afbfSr$ be the algebra homomorphism defined in
\cite[1.11]{Lu00}. According to \cite[3.4(a)]{Lu00} we have
\begin{equation}\label{commute}
\phi_{r+n,r}\circ\dz_{r+n}(x)=\dz_r(x)
\end{equation}
for all $r\in\mbn$ and $x\in\dbfU$.
The following result was proved by Schiffmann--Vasserot \cite{SV} (see also Lusztig \cite[4.1]{Lu00} and Mcgerty \cite[7.10]{Mcgerty}).
\begin{Thm}\label{SV}
$(1)$ We have $\dzr(\dbfBn)\han\{0\}\cup\{\th_{A,r}\mid A\in\afThnr\}$.

$(2)$ For $A\in\afThnrnap$ we have
$$\phi_{r+n,r}(\th_{A,r+n})=
\begin{cases}
\th_{A-E,r}&\text{if $a_{i,i}\geq 1$ for $1\leq i\leq n$},\\
0&\text{otherwise},
\end{cases}$$
where $E=(\dt_{i,j})_{i,j\in\mbz}\in\afThn$.
\end{Thm}

For $A\in\afThnap$ with $A-E\not\in\afThn$ let
$\fb_A=(a_r)_{r\geq 0}\in\prod_{r\geq 0}\afbfSr$, where
$a_r=\th_{A+mE,r}$ if $r=\sg(A)+mn$ for some $m\geq 0$, and $a_r=0$ otherwise.

\begin{Lem}\label{dz(dotbfB)}
We have $\dz(\dbfBn)=\{\fb_A\mid A\in\afThnap,\,A-E\not\in\afThn\}$.
\end{Lem}
\begin{proof}
Let $b\in\dbfBn$. By Theorem \ref{injective} we have $\dz(b)\not=0$.
Let $r_0=\min\{r\in\mbn\mid\dz_{r}(b)\not=0\}$. Then by Theorem \ref{SV}(1) and \cite[8.2]{Lu99} we have $\dz_{r_0}(b)=\th_{A,r_0}$ for some $A\in\afThnrzap$.
From \eqref{commute} we see that $\phi_{r_0,r_0-n}(\th_{A,r_0})=\phi_{r_0,r_0-n}\circ\dz_{r_0}(b)=
\dz_{r_0-n}(b)=0$. Thus by Theorem \ref{SV}(2) we have $A-E\not\in\afThn$. By the proof of \cite[4.3]{Lu00}, we know that if $\dz_r(b)\not=0$ for some $r>r_0$, then $r\equiv r_0\!\!\mod n$. Furthermore, if $m>0$ then by \eqref{commute} we have
$$\th_{A,r_0}=\dz_{r_0}(b)=\phi_{r_0+n,r_0}\circ\phi_{r_0+2n,r_0+n}\circ\cdots\circ\phi_{r_0+mn,
r_0+(m-1)n}\circ\dz_{r_0+mn}(b).$$
This together with Theorem \ref{SV} implies that
$\dz_{r_0+mn}(b)=\th_{A+mE,r_0+mn}$. Thus we have $\dz(b)=\fb_A$.

On the other hand, if $A'\in\afThnap$ with $A'-E\not\in\afThn$,
by \cite[8.2]{Lu99} we conclude that
there exists $b'\in\dbfBn$ such that $\dz_{r_0'}(b')=\th_{A',r_0'}$,
where $r_0'=\sg(A')$. By the proof above we conclude that $\dz(b')=\fb_{A'}$. The assertion follows.
\end{proof}

By Theorem \ref{injective} and Lemma \ref{dz(dotbfB)} we conclude that
for each $A\in\afThnap$ with $A-E\not\in\afThn$, there exists
a unique $\fc_A\in\dbfBn$ such that $\dz(\fc_A)=\fb_A$. Furthermore
we have $$\dbfBn=\{\fc_A\mid A\in\afThnap,\,A-E\not\in\afThn\}.$$
Thus $\dbfBn$ is indexed by the set $\{A\in\afThnap\mid A-E\not\in\afThn\}$.
For $A,B\in\afThnap$ with $A-E,B-E\not\in\afThn$ we write
\begin{equation}\label{cAcB}
\fc_A\fc_B=\sum_{C\in\afThnap\atop C-E\not\in\afThn}\sfh_{A,B,C}\fc_C,
\end{equation}
where $\sfh_{A,B,C}\in\sZ$.

Recall the map $\eta_m$ defined in \eqref{etam}.
The structure constants for the canonical basis $\dbfBn$ of $\dbfU$ and
the structure constants for  the canonical basis $\bfBN^{\text{ap}}=\{\th_A^+\mid A\in\afThNpap\}$ of $\afbfUslNp$ are related in the following way.
\begin{Thm}\label{relation dbfBn bfBNap}
Assume $N>n$. Let $A,B\in\afThnap$ with $A-E,B-E\not\in\afThn$. If $C\in\afThnap$ with
$C-E\not\in\afThn$ is such that
$\sfh_{A,B,C}\not=0$, then there exist $m_1,m_2,m_C\in\mbn$ and $k_0\in\mbz$ such that
$\sg(A)+nm_1=\sg(B)+nm_2=\sg(C)+nm_C$, $\widetilde{A_k},\widetilde{B_k},
\widetilde{C_k}\in\afThNp^{\mathrm{ap}}$ and
$$\sfh_{A,B,C}=\sff_{\widetilde{A_k},\widetilde{B_k},
\widetilde{C_k}}$$
for $k\leq k_0$, where $A_k=\etak(A+m_1E)$, $B_k=\etak(B+m_2E)$, $C_k=\etakt(C+m_CE)$ and
$\sff_{\widetilde{A_k},\widetilde{B_k},
\widetilde{C_k}}$ is as given in \eqref{fABC}.
\end{Thm}
\begin{proof}
By \eqref{cAcB} we have
\begin{equation}\label{bAbB}
\fb_A\fb_B=\sum_{C\in\afThnap\atop C-E\not\in\afThn}\sfh_{A,B,C}\fb_C,
\end{equation}
where $\sfh_{A,B,C}\in\sZ$.
If $\sg(A)\not\equiv\sg(B)\!\!\mod n$ then by definition we have $\fb_A\fb_B=0$. Now we assume $\sg(A)\equiv\sg(B)\!\!\mod n$.
Let $\sX=\{C\in\afThnap\mid C-E\not\in\afThn,\,
\sfh_{A,B,C}\not=0\}$. We choose $r_0\in\mbn$ such that $r_0\equiv\sg(A)\!\!\mod n$, $r_0\geq\sg(A)$, $r_0\geq\sg(B)$ and $r_0\geq\sg(C)$ for $C\in\sX$.
Note that  $\sg(C)\equiv\sg(A)\!\!\mod n$ for $C\in\sX$.
Assume $r_0=\sg(A)+nm_1=\sg(B)+nm_2=\sg(C)+nm_C$ for $C\in\sX$.
Then by \eqref{bAbB} we have
$$\th_{A+m_1E,r_0}\th_{B+m_2E,r_0}=\sum_{C\in\sX}
\sfh_{A,B,C}\th_{C+m_CE,r_0}.$$
This implies that $\sfh_{A,B,C}=\g_{A+m_1E,B+m_2E,C+m_CE,r_0}$. Now the assertion follows from
Theorem \ref{prop of canonical basis for affine q-Schur algebras}.
\end{proof}

The following theorem gives the positivity property for $\dbfU$.
\begin{Thm}\label{positive modified affine sln}
For $b,b'\in\dbfBn$ we have $bb'\in\sum_{b''\in\dbfBn}\mbn[\up,\up^{-1}]b''$.
\end{Thm}
\begin{proof}
The assertion follows from Theorem \ref{positive affine sln} and Theorem \ref{relation dbfBn bfBNap}.
\end{proof}
\section{A weak positivity property for $\ddbfHa$}
For $\la,\mu\in\afmbzn$ we set ${}_\la\dbfHa_\mu=\dbfHa/{}_\la I_\mu$, where
\begin{equation*}\label{laImu}
{}_\la I_\mu=\big(\sum_{\bfj\in\afmbzn}(\Kbfj-
 \up^{\la\cdot\bfj})\dbfHa+\sum_{\bfj\in\afmbzn}\dbfHa(\Kbfj
 -\up^{\mu\cdot\bfj})\big).
 \end{equation*}
 Let
$\ddbfHa:=\bop_{\la,\mu\in\afmbzn}{}_\la\dbfHa_\mu.$ As in the case of $\dbfU$, there is a natural associative $\mbq(\up)$-algebra structure on $\ddbfHa$ inherited from that of $\dbfHa$ (see \cite{Fu13}). The algebra $\ddbfHa$ is the modified form of $\dbfHa$.
Let $\{\th_A\mid A\in\aftiThn\}$ be the canonical basis of $\ddbfHa$ defined in \cite{DF14},
where $\aftiThn$ is given in \S 1.
\begin{Prop}[{\cite[7.7]{DF14}}]\label{dxr}
There is a surjective algebra homomorphism $\dxr:\ddbfHa\to\afbfSr$ such that
$$\dxr(\th_A)=\begin{cases}\th_{\Ar}, &\text{ if } A\in\afThnr;\\
0, &\text{ otherwise.}
\end{cases}$$
\end{Prop}
The maps $\dxr$ induce an algebra homomorphism
$$\dx:\ddbfHa\ra\prod_{r\geq 0}\afbfSr$$
such that $\dx(x)=(\dxr(x))_{r\geq 0}$ for $x\in\ddbfHa$.  Contrast to Theorem \ref{injective}, the map $\dx$ is not injective.
For $A\in\aftiThn$ let $\ol{\th_A}=\th_A+\ker(\dx)\in\ddbfHa/\ker(\dx)$.

\begin{Lem}\label{ol thA}
We have $\ol{\th_A}=0$ for $A\not\in\afThn$ and the set $\{\ol{\th_A}\mid A\in\afThn\}$ forms a $\mbq(\up)$-basis
for $\ddbfHa/\ker\dx$.
\end{Lem}
\begin{proof}
From Proposition \ref{dxr} we see that $\ker\dx=\spann_{\mbq(\up)}\{\th_A\mid A\in\aftiThn,\,A\not\in\afThn\}$. The assertion follows.
\end{proof}

The following result gives a weak version of the positivity property for $\ddbfHa$.
\begin{Thm}\label{positive modified affine gln}
For $A,B\in\afThn$ we have $\ol{\th_A}\cdot\ol{\th_B}\in\sum_{C\in\afThn}\mbn[\up,\up^{-1}]\ol{\th_C}$.
\end{Thm}
\begin{proof}
The assertion follows from Corollary \ref{positive affine Schur}, Proposition \ref{dxr} and Lemma \ref{ol thA}.
\end{proof}


\begin{thebibliography}{99}\frenchspacing
\bibitem[BLM]{BLM}
A.~A. Beilinson, G. Lusztig and R. MacPherson, {\em A geometric
setting for the quantum deformation of $GL_n$}, Duke Math.J. {\bf
61} (1990), 655--677.


\bibitem[C]{Curtis}
C. W. Curtis, {\em On Lusztig¡¯s isomorphism theorem for Hecke algebras}, J. Algebra {\bf 92} (1985), 348--365.

\bibitem[DDPW]{DDPW}
B. Deng, J. Du, B. Parshall and J.-p. Wang,
{\em Finite dimensional algebras and quantum groups}, Mathematical Surveys and Monographs, Vol. {\bf 150},
Amer. Math, Soc. (2008).

\bibitem[DDF]{DDF}
B. Deng, J. Du and Q. Fu,
{\em A double Hall algebra approach to affine quantum Schur--Weyl theory}, London
Mathematical Society Lecture Note Series, {\bf 401}, Cambridge University Press, 2012.


\bibitem[D]{Du92}
J. Du, {\em Kahzdan-Lusztig bases and isomorphism
 theorems for $q$-Schur algebras}, Contemp. Math. {\bf 139} (1992),
 121--140.

\bibitem[DF1]{DF10}
J. Du and Q. Fu,
{\em A modified BLM approach to quantum affine
$\frak{gl}_n$}, Math. Z. {\bf 266} (2010), 747--781.

\bibitem[DF2]{DF13}
J. Du and Q. Fu, {\em Quantum affine $\frak{gl}_n$ via Hecke algebras,}
preprint, arXiv:1311.1868.

\bibitem[DF3]{DF14}
J. Du and Q. Fu,
{\em The Integral quantum loop algebra of $\mathfrak{gl}_n$,} preprint,
arXiv:1404.5679.

\bibitem[F1]{Fu13}
Q. Fu, {\em Integral affine Schur--Weyl reciprocity}, Adv. Math. {\bf 243} (2013), 1--21.

\bibitem[F2]{Fu14}
Q. Fu, {\em Canonical bases for modified quantum $\frak{gl}_n$ and $q$-Schur algebras}, J. Algebra, {\bf 406} (2014), 308--320.


\bibitem[GV]{GV}
V. Ginzburg and E. Vasserot, {\em Langlands reciprocity for affine quantum groups of type $A_n$}, Internat. Math. Res. Notices 1993, 67--85.

\bibitem[G1]{Gr97}
R. M. Green, {\em
Positivity properties for $q$-Schur algebras}, Math. Proc. Cambridge Philos. Soc.  {\bf 122}  (1997), 401--414.


\bibitem[G2]{Gr99}
R. M. Green, {\em The affine $q$-Schur algebra}, J. Algebra {\bf 215} (1999),  379--411.



\bibitem[K]{Kas1}  M. Kashiwara, {\em On crystal bases of the $q$-analogue of universal enveloping algebras}, Duke Math. J. {\bf 63} (1991), 465--516.


\bibitem[KL]{KL79}
D. Kazhdan and G. Lusztig, {\em Representations of Coxeter groups and Hecke Algebras,} Invent. Math. {\bf 53} (1979), 165--184.

\bibitem[L1]{Lu90b}
G. Lusztig, \textit{Canonical bases arising from quantized
 enveloping algebras,} J. Amer. Math. Soc. {\bf 3} (1990), 447--498.

\bibitem[L2]{Lu91}
G. Lusztig, \textit{Quivers, perverse sheaves, and quantized enveloping algebras}, J. Amer. Math. Soc. {\bf 4} (1991), 365--421.


\bibitem[L3]{Lu92}
G. Lusztig, {\em Affine quivers and canonical bases,}
Inst. Hautes $\mathrm{\acute{E}}$tudes Sci. Publ. Math. {\bf 76}  (1992), 111--163.

\bibitem[L4]{Lu92b}
G. Lusztig, {\em Canonical bases in tensor products,} Proc. Nat. Acad. Sci. U.S.A. {\bf 89} (1992), 8177--8179.

\bibitem[L5]{Lubk}
G. Lusztig, {\em Introduction to quantum groups},
Progress in Math. {\bf 110}, Birkh\"auser, 1993.

\bibitem[L6]{Lu99}
G. Lusztig, {\em Aperiodicity in quantum affine $\frak{gl}_n$},
Asian J. Math. {\bf 3} (1999),  147--177.

\bibitem[L7]{Lu00} G. Lusztig, {\em Transfer maps for quantum affine ${\frak {sl}}_n$},
in: Representations and quantizations (Shanghai, 1998), China High.
Educ. Press, Beijing, 2000, 341--356.

\bibitem[M]{Mcgerty}
K. McGerty, {\em On the geometric realization of the inner product and canonical basis for quantum affine $\frak{sl}_n$}, Algebra Number Theory  {\bf 6}  (2012),  1097--1131.

\bibitem[R]{Ri93}
C.~M. Ringel, {\em The composition algebra of a
cyclic quiver}, Proc. London Math. Soc. {\bf 66} (1993), 507--537.

\bibitem[SV]{SV}
O. Schiffmann and E. Vasserot, {\em Geometric construction of the global base of the quantum modified algebra of $\h{\frak{gl}}_n$}, Transform. Groups  {\bf 5}  (2000), 351--360.



\bibitem[VV]{VV99} M. Varagnolo and E. Vasserot, \textit{On the
decomposition matrices of the quantized Schur algebra}, Duke Math.
J. {\bf 100} (1999), 267--297.

\bibitem[X]{X97}
J. Xiao, {\em Drinfeld double and Ringel-Green theory of Hall algebras},
J. Algebra {\bf 190} (1997), 100--144.
\end{thebibliography}
\end{document}